\documentclass[10pt]{amsart}

\usepackage{enumerate,url,amssymb,  mathrsfs,  graphicx, pdfsync}
\newtheorem{theorem}{Theorem}[section]
\newtheorem{lemma}[theorem]{Lemma}
\newtheorem*{lemma*}{Lemma}

\newtheorem{proposition}[theorem]{Proposition}
\newtheorem{corollary}[theorem]{Corollary}

\theoremstyle{definition}
\newtheorem{definition}[theorem]{Definition}

\newtheorem{question}[theorem]{Question}

\theoremstyle{remark}

\numberwithin{equation}{section}

\newcommand{\C}{\mathbb{C}}

\newcommand{\W}{\mathscr{W}}

\newcommand{\X}{\mathbb{X}}

\newcommand{\Y}{\mathbb{Y}}

\newcommand{\dtext}{\textnormal d}

\newcommand{\onto}{\xrightarrow[]{{}_{\!\!\textnormal{onto\,\,}\!\!}}}

\newcommand{\deff}{\stackrel {\textnormal{def}}{=\!\!=} }
\DeclareMathOperator{\diam}{diam}

\def\le{\leqslant}

\begin{document}

\title[Approximation up to the boundary of homeomorphisms]{Approximation up to the Boundary of Homeomorphisms of Finite Dirichlet Energy}

\author[Iwaniec]{Tadeusz Iwaniec}
\address{Department of Mathematics, Syracuse University, Syracuse,
NY 13244, USA and Department of Mathematics and Statistics,
University of Helsinki, Finland}
\email{tiwaniec@syr.edu}

\author[Kovalev]{Leonid V. Kovalev}
\address{Department of Mathematics, Syracuse University, Syracuse,
NY 13244, USA}
\email{lvkovale@syr.edu}

\author[Onninen]{Jani Onninen}
\address{Department of Mathematics, Syracuse University, Syracuse,
NY 13244, USA}
\email{jkonnine@syr.edu}
\thanks{ T. Iwaniec was supported by the NSF grant DMS-0800416 and the Academy of Finland project 1128331. L. Kovalev was supported by the NSF grant DMS-0968756.
J. Onninen was supported by the NSF grant DMS-1001620.}

%    General info

\date{\today}


\begin{abstract}
Let $\,\mathbb X \subset \mathbb C\,$ and $\,\mathbb Y \subset \mathbb C\,$  be Jordan domains of the same finite connectivity, $\,\mathbb Y\,$ being inner chordarc regular (such are Lipschitz domains). Every homeomorphism $h \colon \mathbb X \to \mathbb Y\,$ in the Sobolev space
$\W^{1,2}$ extends to a continuous map $h \colon \overline{\mathbb X} \to \overline{\mathbb Y}\,$.  We prove that there exist homeomorphisms $\,h_k \colon \overline{\mathbb X} \to \overline{\mathbb Y}\,$  which converge to $\,h\,$ uniformly and in $\,\mathscr W^{1,2}(\mathbb X\,, \mathbb Y )\,$.  The problem of approximation of Sobolev homeomorphisms, raised by J. M. Ball and L. C. Evans, is deeply rooted in a study of energy-minimal deformations in nonlinear elasticity. The new feature of our main result is that approximation takes place also on the boundary, where the original map need not be a homeomorphism.
\end{abstract}

\maketitle

\section{Introduction}
Throughout this text $\,\mathbb X\,$ and $\,\mathbb Y\,$ are finitely connected Jordan domains in the complex plane $\,\mathbb C \simeq \mathbb R^2$. Each boundary $\, \partial \X\, $ and $\,\partial \Y\, $ consists of $\,\ell\,$  disjoint Jordan curves.
We shall consider orientation preserving homeomorphisms $\, h \, : \mathbb X \onto\,\mathbb Y\,$ of finite Dirichlet energy
$$
\mathscr E_{{_\mathbb X}}[h] = \iint_\mathbb X |\,Dh(z)|^2\; \dtext x \,\dtext y \;<\;\infty, \qquad z=x+iy
$$
where $|Dh|$ stands for the Hilbert-Schmidt norm of the derivative matrix $Dh$. 
Such a mapping has a continuous extension $h \colon \overline{\X} \onto \overline{\Y}$ which is not necessarily a homeomorphism~\cite{IO}. In this paper we show that $h$ can be strongly approximated by homeomorphisms between closed domains, provided $\Y$ is inner chordarc regular, see Definition~\ref{ICA}. In particular, $\Y$ can be a Lipschitz domain.

\begin{theorem}\label{Main}
Let $\X$ and $\Y$ be finitely connected Jordan domains, $\Y$ being inner chordarc.   Let  $\, h\,:\,\mathbb X \onto \mathbb Y\,$  be a homeomorphism in the Sobolev space $\,\mathscr W^{1,2}(\mathbb X ,\,\mathbb Y ) \,$. Then:
 \begin{enumerate}[(a)]
 \item  there exist  homeomorphisms $\, h_k\,:\,\overline{\mathbb X} \onto \overline{\mathbb Y}\,,\; k =1, 2, \dots,\,$ that converge to $\, h\,:\,\overline{\mathbb X} \onto \overline{\mathbb Y}\,$  uniformly and strongly in $\,\mathscr W^{1,2}(\mathbb X )\,$.
\item Moreover, for every compact subset $\,\mathbb G \subset \mathbb X\,$, we have  $\,h_k \equiv h\,$ on $\,\mathbb G \,$, provided $\,k = k(\mathbb G)\,$ is sufficiently large.
 \item If, in addition,  $\, h\,:\,\mathfrak X \rightarrow  \partial \,\mathbb Y\,$ is injective on a compact subset $\,\mathfrak X \subset \partial\,\mathbb X\,$, then $\, h_k\,$ can be chosen so that $\,h_k \equiv h\,$ on $\,\mathfrak X $.
 \end{enumerate}
\end{theorem}

The motivation for part (c) comes from variational problems for mappings between quadrilaterals, i.e., Jordan domains with four distinguished points which form the injectivity set $\mathfrak X$. We will pursue such applications of Theorem~\ref{Main} elsewhere. 

Let us compare Theorem~\ref{Main} to the known results in this field. The problem of approximation of Sobolev homeomorphisms was raised by J.~ M.~Ball and L.~C.~Evans in 1990s~\cite{Ba,Ba2}. There have been several recent advances in this direction~\cite{BM, DP, IKO1, IKO2, Mo}. In particular, in~\cite{IKO1, IKO2} we proved that any $\W^{1,p}$-homeomorphism $h\colon \mathbb U\onto\mathbb V$ between open subsets of $\mathbb R^2$ can be approximated in the $\W^{1,p}$ norm by diffeomorphisms of these open subsets. 
The new feature of Theorem~\ref{Main} is that approximation takes place also on the boundary, where the original map need not be a homeomorphism. This explains why Theorem~\ref{Main} imposes regularity assumptions on the boundaries of $\X$ and $\Y$. 

Combining Theorem~\ref{Main} with~\cite[Theorem 1.2]{IKO1} yields the following result. 

\begin{corollary}
If  $\,h\,:\,\mathbb X \onto \mathbb Y\,$  is a $\,\mathscr C^\infty$--diffeomorphism in $\mathscr W^{1,2}(\X)$ then, in addition to all properties listed in 
Theorem~\ref{Main}, the  mappings $\,h_k \,:\,\mathbb X \onto \mathbb Y\,$ can also be found as $\,\mathscr C^\infty$ -diffeomorphisms. If $h$ is only a homeomorphism, such an approximation by diffeomorphisms is still available, except for the property (b). 
\end{corollary}

One of the fundamental problems in topology is to approximate continuous mappings by homeomorphisms. The approximation procedures, still only partially understood, have led topologists to the concept of monotone mappings~\cite{Mor} and somewhat subtle concept of cellular mappings~\cite{Br}. We refer the interested reader to~\cite{Mc, Rab}. In the mathematical theory of hyperelasticity, on the other hand, we are concerned with the energy-minimal deformations $h \colon \X \onto \Y$, so having additional Sobolev type regularity. However, very often the injectivity of the energy minimal mappings is lost, though they enjoy some  features of homeomorphisms, like monotonicity.  In particular, the question of approximation of a monotone mapping in the Sobolev space $\mathscr W^{1,2}(\X)$  by homeomorphisms $h_k \colon \overline{\X} \onto \overline{\Y}$ gains  interest in the mathematical models of elasticity. A novelty in these directions is the following corollary of Theorem~\ref{Main} and~\cite[Theorem~1.6]{CIKO}.

\begin{corollary} 
Suppose $\,h\,:\,\mathbb X \onto \mathbb Y\,$  lies in the Sobolev space $\mathscr W^{1,2}(\X)$ and extends 
to a continuous monotone map $\,h\,:\, \overline{\mathbb X} \onto \overline{\mathbb Y}$. Then
there exist  homeomorphisms $\, h_k\,:\,\overline{\mathbb X} \onto \overline{\mathbb Y}\,,\; k =1, 2, \dots,\,$ that converge to
 $\, h\,:\,\overline{\mathbb X} \onto \overline{\mathbb Y}\,$  uniformly and strongly in $\,\mathscr W^{1,2}(\mathbb X )\,$.
 \end{corollary}

\begin{question}
Does Theorem~\ref{Main} remain valid when target $\Y$ is an arbitrary finitely connected Jordan domain?
\end{question}

\begin{question}
Can Theorem~\ref{Main} be extended to Sobolev spaces $\W^{1,p}$, $p\in (1, \infty)$? Or to dimensions $n>2$? 
\end{question}

\section{Preliminaries}

 \subsubsection{Royden Algebra}
The {Royden algebra} $\,\mathscr R(\mathbb X)\,$ consists of  continuous functions  $g \colon \overline{\mathbb X} \to \C$  which have finite energy. The norm is given by
$$ \big{\|} g \big{\|}_{\mathscr R(\mathbb X)}\, = \,\underset{z\in \mathbb X}{\sup}\,|g(z)| \,+\,\Big(\iint_\mathbb X |Dg(z)|^2 \;\dtext x \,\dtext y \Big)^{\frac{1}{2}}$$
 Let  $\,\mathscr R_0(\mathbb X)\,$ denote the completion of  $\,\mathscr C^\infty_0 (\mathbb X)\,$ in this norm. Any conformal mapping $\,\varphi \colon \X' \onto \X\,$ between Jordan domains  induces an isometry $\varphi^{\,\sharp} \colon \mathscr R(\mathbb X) \onto \mathscr R(\mathbb X')$ by the rule $\,\varphi^{\,\sharp} (h)=h\circ \varphi\,$. Thus the domain of definition of $\,h\,$ is permitted to be changed by any conformal transformation. The following observation will allow us to transform the target.

\begin{lemma}\label{folklore}
 Let $\,h, h_k \colon \X \onto \Y\,$ be homeomorphisms in $\,\mathscr R(\mathbb X)\,$ and $\,\Phi \colon {\Y} \onto {\Y}'\,$ be a $\,\mathscr C^1\,$-diffeomorphism that extends to a homeomorphism $\,\Phi \colon \overline{\Y} \onto \overline{\Y}'\,$. Assume that
  both the gradient matrix $\,D\Phi\,$ and its inverse $\,(D\Phi)^{-1}\,$ are bounded in $\,\Y\,$.  Then $\,\Phi \circ h_k \to \Phi \circ h\,$ in  $\,\mathscr R(\mathbb X)\,$ if and only if $\,h_k \to h\,$ in  $\,\mathscr R(\mathbb X)\,$.
 \end{lemma}

 Our primary appliance for strong approximation in Theorem \ref{Main} will be local harmonic replacements near the boundaries of $\,\mathbb X\,$. We will rely on a well-known fact, see \cite{IKO1}.

 \begin{lemma}
 In a finitely connected Jordan domain $\, \mathbb X \subset \mathbb C\,$, we consider a function $g \in \mathscr R(\mathbb X)\,$. Let $\,h :\overline{\mathbb X} \rightarrow \mathbb C\,$  denote the continuous harmonic extension of the boundary map $\,g\,:\partial \,\mathbb X \rightarrow \mathbb C\,$ into $\,\mathbb X\,$. Then $\,h \in g +  \mathscr R_\circ(\mathbb X)\,$. Moreover,
            \begin{equation}\label{Energy}
            0 \;\leqslant \;\mathscr E_{_\mathbb X} [ g - h ]\; = \,\,\mathscr E_{_\mathbb X} [g]\;-\; \mathscr E_{_\mathbb X} [ h ] \;\leqslant \; \mathscr E_{_\mathbb X} [g]\,.
            \end{equation}
            \end{lemma}

\subsubsection{Inner Chordarc Domains}

 \begin{definition}\label{ICA}
A finitely connected Jordan domain $\Y$ is \textit{inner chordarc}  if there exists a constant $C$ with the following property. Suppose that $a,b$ belong to the same boundary component of $\Y$  and $\gamma \subset \Y$ is  an open Jordan arc with endpoints at $\,a\,$ and $\,b\,$. Then the shortest connection from $a$ to $b$ along $\partial \Y$ has length at most $\,C \cdot\textnormal{length}({\gamma})$.
 \end{definition}

 Inner chordarc domains were studied in Geometric Function Theory since 1980s~\cite{FH, HS, La, Po, Tu, Va, Va2}.  They are more general than Lipschitz domains. For instance they allow inward cusps and logarithmic spiraling, see Figure~\ref{exa1}.

 \begin{theorem}\label{thmICA}\cite{Va}
A simply connected Jordan domain   $\,\Omega\,$ is inner chordarc if and only if there exists a  $\,\mathscr C^1\,$-diffeomorphism $F$ from $\Omega$ onto the unit disk $\mathbb D$  that extends to a homeomorphism $\,F \colon \overline{\Omega} \onto \overline{\mathbb D}\,$ such that  both gradient matrices $\,DF\,$ and $\,(DF)^{-1}\,$ are bounded in $\Omega$.
 \end{theorem}

\begin{center}\begin{figure}[h]
\includegraphics[width=0.9\textwidth]{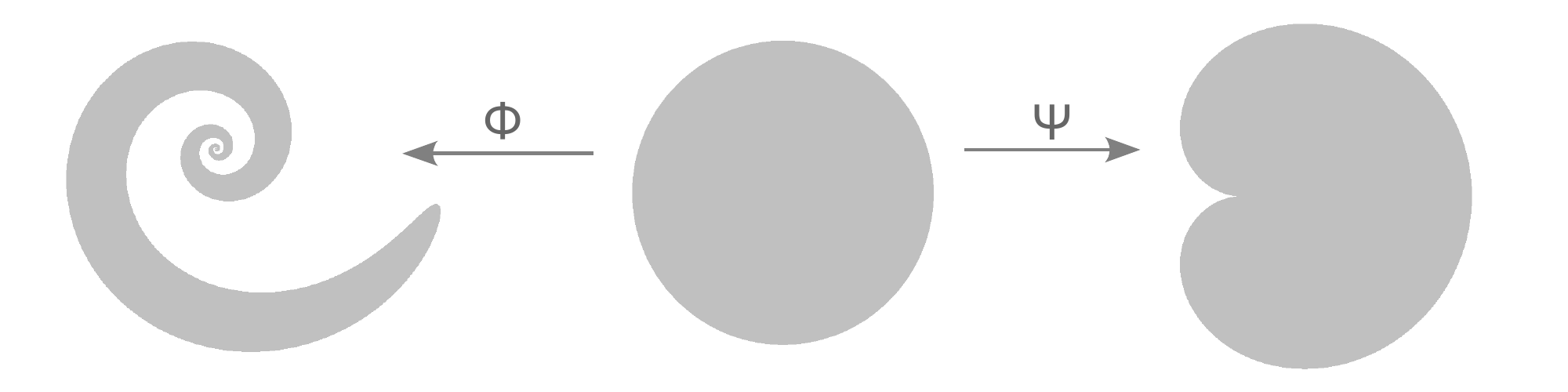} \caption{Inner chordarc domains}\label{exa1}
\end{figure}\end{center}

Fugure~\ref{exa1} illustrates two such mappings from the unit disk onto a non-Lipschitz domain:  
$\Phi(z)=|z-1|^{4i}(z-1)$ (the image is a spiral domain) and $\Psi(z)=(z+1)^2/|z+1|$ (the image contains an inward cusp).

\subsubsection{Monotone mappings}

 Recall that a continuous map $\, h \, : \overline{\mathbb X} \onto\,\overline{\mathbb Y}\,$ is monotone if the preimage  of every continuum (connected compact set) in $\,\overline{\mathbb Y}\,$ is a continuum in $\overline{\mathbb X}\,$.  A point $\,y_\circ \in \overline{\Y}\,$ is said to be a \textit{simple value} of $\,h\,$ if its preimage  $\,h^{-1}\{y_\circ\}\,$ is a single point in $\,\overline{ \mathbb X}\,$.

 We need an elementary lemma.

 \begin{lemma}\label{Monotone}
 Let   $\,h\, :\,\mathbb X \onto \mathbb Y\,$ be a homeomorphism between $\,\ell\,$-connected Jordan domains and suppose that it extends continuously up to the boundary. Then  the extended map  $\,h\, :\,\overline{\mathbb X} \onto \overline{\mathbb Y}\,$ is  monotone. Furthermore, the inverse map $\,h^{-1} :\,\mathbb Y \onto \mathbb X\,$ extends continuously up to  simple values of $\,h\,$ which are virtually all points in $\,\partial \mathbb Y\,$, except for a countable number of them.
  \end{lemma}

 Every  homeomorphism with finite Dirichlet energy has a continuous extension to the boundary~\cite{IO}. Precisely, 

\begin{theorem}\label{Extension}
Every finite energy homeomorphism $\,h\, :\,\mathbb X \onto \mathbb Y\,$ between $\,\ell$ -connected Jordan domains extends to a continuous map between the closures, again denoted by $\,h\, :\,\overline{\mathbb X} \onto \overline{\mathbb Y}\,$. This map is monotone, though the inverse $\,h^{-1}\, :\,\mathbb Y \onto \mathbb X\,$ may not admit continuous extension to the closure of $\,\mathbb Y\,$, unless it also has finite energy.
\end{theorem}

\subsubsection{Harmonic Homeomorphisms}

We shall make use of the following strong version of the Rad\'{o}-Kneser-Choquet theorem, see~\cite[\S 3.2]{Dub}.
\begin{theorem}[Rad\'{o}-Kneser-Choquet]\label{RKC}
Let $\,\mathbb U \subset \mathbb C \,$ be a simply connected Jordan domain and $\,\Omega \subset \mathbb C\,$ a bounded convex domain. Suppose we are  given a continuous monotone map $\,h : \partial \,\mathbb U \onto \partial \,\Omega\,$, not necessarily a homeomorphism.  Then its continuous harmonic extension, denoted by $\,H : \overline{\mathbb U} \rightarrow \mathbb C\,$, defines a $\,\mathscr C^\infty$ -diffeomorphism $\,H : \mathbb U \onto\Omega\,$.
\end{theorem}

 \section{Proof of Theorem \ref{Main}}

\begin{center}\begin{figure}[h]
\includegraphics[width=0.33\textwidth]{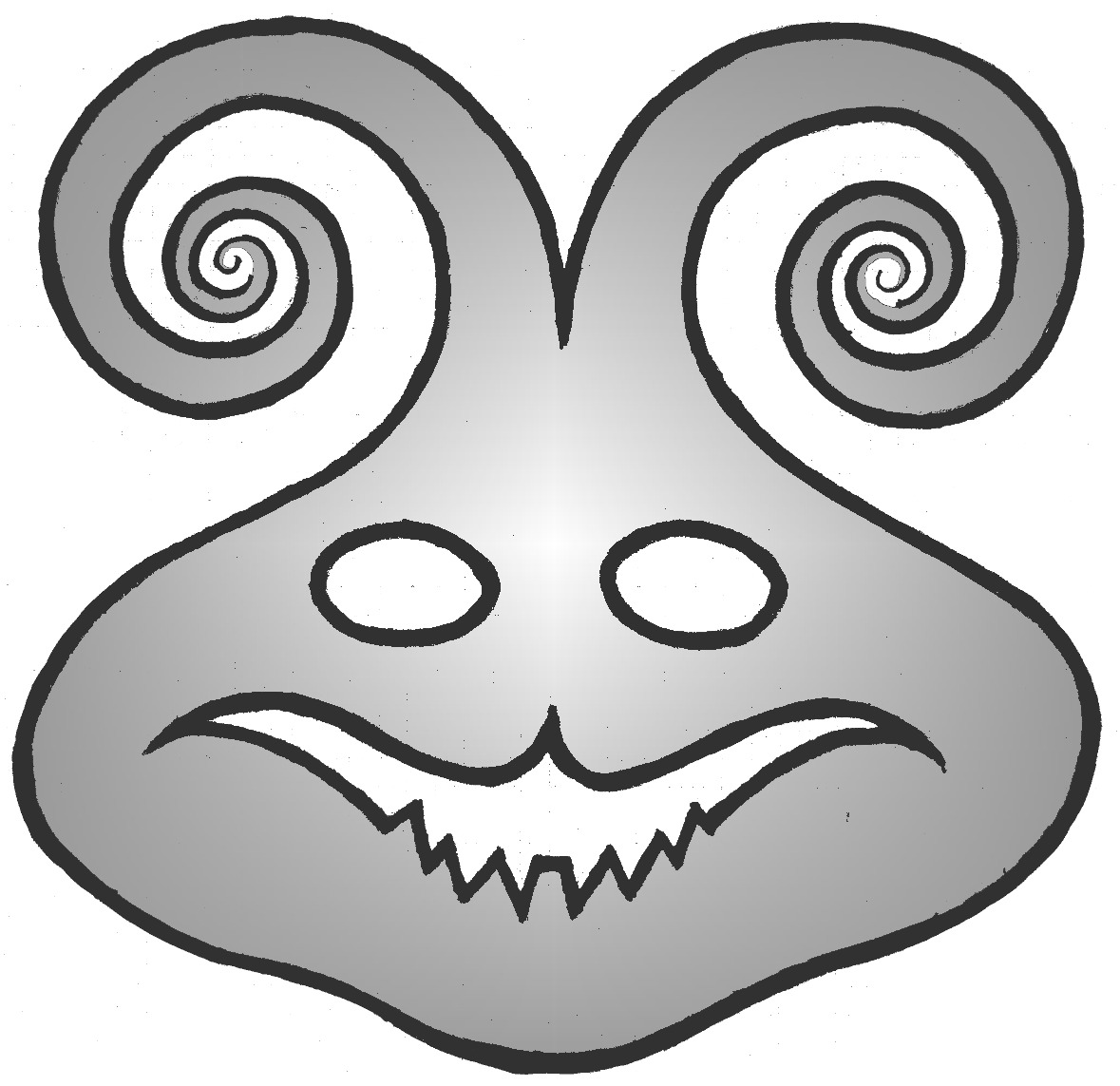} \caption{A finitely connected inner chordarc domain that is not Lipschitz}\label{ram}
\end{figure}\end{center}

We reserve the notation,
$$\mathfrak X_1, \mathfrak X_2, \dots, \mathfrak X_\ell \;,\;\;\;\;\;\; \textnormal{for the components of}  \;\;\partial \X $$
$$\Upsilon_1, \Upsilon_2, \dots, \Upsilon _\ell  \;,\;\;\;\;\;\textnormal{for the components of} \;\; \partial \Y$$
The components are numbered so that $h$ sends $\mathfrak X_\nu$ onto $\Upsilon_\nu$ for $\nu=1,\dots,\ell$.
Figure~\ref{ram} illustrates a domain $\Y$ that satisfies the assumptions of the theorem.

We recall that the boundary map $\, h : \partial \,\mathbb X \onto \partial\,\mathbb Y\,$
is injective on a compact subset $\,\mathfrak X\ \subset \partial\,\mathbb X = \mathfrak X_1 \cup \mathfrak X_2 \cup \dots \cup \mathfrak X_\ell\,$, possibly empty.
Let us fix an $\,\varepsilon > 0\,$ and  a compact $\mathbb G \subset \mathbb X\,$.
We need to construct a homeomorphism $\,h_\varepsilon :\overline{\mathbb X} \onto \overline{\mathbb Y}\,$ of Sobolev class $\,\mathscr W^{1,2}(\mathbb X , \mathbb Y)\,$ which coincides with $\,h\,$ on $\,\mathfrak X \cup \mathbb G\,$ and
\begin{equation}\label{H}
 \| h_\varepsilon - h \|_{\mathscr R(\mathbb X)}\; \preccurlyeq \varepsilon\;
\end{equation}
 
Hereafter the symbol $\,\preccurlyeq\,$ indicates that the inequality holds with an \emph{implied} multiplicative constant.
The implied constants will vary from line to line but remain \textbf{independent of $\,\varepsilon\,$} as long as $\,\varepsilon\,$ is sufficiently small. 

To prove~\eqref{H} we shall set up a chain of homeomorphisms $\,h^0 ,\, h^1 ,\dots,\, h^\ell \, : \mathbb X \onto \mathbb Y\,$  in the Sobolev space $\,\mathscr W^{1,2}(\mathbb X , \mathbb Y )\,$ whose continuous extensions, still denoted by $\,h^0 ,\, h^1 ,\dots,\, h^\ell \, : \overline{\mathbb X} \onto \overline{\mathbb Y}\,$, satisfy\\

 \begin{itemize}
 \item $\,h^0 \equiv h\,$\\
 \item $\, h^1 \,: \mathfrak X_1 \cup \mathfrak X \,\rightarrow \partial\,\mathbb Y\,$ is injective,\\ \;\; $\,h^1 \equiv h\,$  on  $\,\mathfrak X\cup\mathbb G\,$, \\
     $
     \| h^1 - h^0 \|_{\mathscr R(\mathbb X)}\; \preccurlyeq \varepsilon\,$\\
      \item For $2\le \nu \le \ell$,  \\ 
      $\, h^\nu \,: \mathfrak X_1 \cup \dots \cup \mathfrak X_\nu \cup\mathfrak X \,\rightarrow \partial\,\mathbb Y\,$ is injective,\\ \;\; $\,h^\nu \equiv h^{\nu - 1}\,$  on  $\,\mathfrak X_1\cup \dots \cup \mathfrak X_{\nu - 1}\cup \mathfrak X \cup \mathbb G\,$, \\
     $
     \| h^\nu - h^{\nu - 1} \|_{\mathscr R(\mathbb X)}\; \preccurlyeq \varepsilon
     $
      \end{itemize}
Thus the final term $\,h^\ell\,$ works for the desired homeomorphism $\,h_\varepsilon :\overline{\mathbb X} \onto \overline{\mathbb Y}\,$, by the triangle inequality.  We proceed by induction on~$\nu$.
The induction begins with $h^0$, which is obvious.  Suppose we are given the map $\,h^{\nu} :\overline{\mathbb X} \onto \overline{\mathbb Y}$ for some 
$0\leqslant \nu < \ell \,$. The construction of  $\,h^{\nu+1} :\overline{\mathbb X} \onto \overline{\mathbb Y}\,$ will be made in 5 steps.\\

\textbf{Step I.} (\textit{Transition to the case when $\,\Upsilon_{\nu+1} = \mathbb T\,$ is the outer boundary}).
To make this transition rigorous, let us perform the following transformations of $\,\overline{\mathbb Y}\,$. First, we reduce ourselves to the case in which  $\,\Upsilon_{\nu+1}\,$ is the outer boundary of $\,\mathbb Y\,$ by applying an inversion if necessary. Once  $\, \Upsilon _{\nu+1}$ is the outer boundary of $\mathbb Y$  we apply Theorem \ref{thmICA} to transform the bounded component  of $\C\setminus\, \Upsilon _{\nu+1}\,$, denoted by $\,\Omega\,$,   onto the unit disk. Let this transformation be denoted by  $F : \Omega \onto \mathbb D$. This map extends to a homeomorphism  $F : \overline{\Omega} \onto \overline{\mathbb D}$, and has both matrix functions $\,DF\,$ and $\,(DF)^{-1}\,$ continuous and bounded in $\Omega$. Let $\Y'=F(\Y)$. 
By virtue of Lemma~\ref{folklore} the composition with
 the mapping $F\colon\Y\onto \Y'$ transforms converging sequences of mappings $h_k \colon \X\onto \Y$ into converging sequences of mappings $F\circ h_k \colon \X\onto\Y'$. The inverse $F^{-1}$ has the same property; therefore we can work with the target $\Y'$ instead of $\Y$. In what follows we simply assume that $\,\Upsilon _{\nu+1}  = \mathbb T\,$ is the outer boundary of $\,\mathbb Y\,$. \\

\textbf{Step II.} (\textit{Harmonic Replacements Near} $\,\mathfrak X_{\nu+1}\,$).
The idea  is to alter $\,h^{\nu} \,$ in a thin neighborhood of  $\,\mathfrak X_{\nu+1}\,$ to gain  piecewise harmonicity therein. In this step we change  neither the boundary map $\,h^{\nu} \,: \partial\, \mathbb X \onto \partial\, \mathbb Y\,$ nor the  values of $h^\nu$ on the given compact  $\mathbb G \subset \X$.

Recall from Lemma \ref{Monotone} that all but a countable number of points in $\, \mathbb T\,$ are simple values of the map $\,h^{\nu} : \mathfrak X_{\nu+1} \onto \mathbb T\,$. They are dense, so one can  partition $\,\mathbb T\,$ into arbitrarily small \textit{closed} circular arcs whose ends are simple values of $\,h^{\nu} \,$,
$$\;
\mathbb T  \;=\;
   \bigcup^N_{\kappa = 1}\, \mathcal C_\kappa \;,\;\;\;\; \diam \mathcal C_\kappa \leqslant \varepsilon <2 
    \;,\;\;\;\;\;\;\textnormal{for all}\;\;\kappa = 1, \dots, N
   \,$$
Let  $\,\mathbb S_\kappa\,$  denote the segment of the unit disk; the open region between  the arc  $\,\mathcal C_\kappa\,$ and the closed line interval $\, \mathbf I_\kappa\,$ connecting the endpoints of $\mathcal C_\kappa$, which we call the base of $\mathbb S_\kappa$. We require the partition of $\,\mathbb T\,$ to be fine enough so that the compact set $\, h^{\nu} (\mathbb G)\,$ intersects none of the segments $\,\mathbb S_\kappa\,, \;\kappa = 1,\dots,N\,$. One further restriction on the partition comes from the following observation: the finer the partition the closer to $\,\mathbb T\,$ are the segments $\,\mathbb S_\kappa\,$. Since $\,h^{\nu} : \mathbb X \onto \mathbb Y\,$ is a homeomorphism, it follows that the preimages of $\,\mathbb S_\kappa\,$ under $\,h^{\nu} \,$, denoted by $\mathbb X_\kappa\,$,  can be  as close to $\,\mathfrak X_{\nu+1}\,$ as we wish. In particular, we may ensure that
\begin{equation}\label{ENergy}
\sum_{\kappa = 1}^N  \iint_{\mathbb X_\kappa} |D h^{\nu}|^2  \; \leqslant \varepsilon^{\,2}
\end{equation}
That is all what we require to determine the partition of $\mathbb T$. This partition will remain fixed for the rest of the proof.  Now we observe that each $\,\mathbb X_\kappa\,$ is a simply connected  Jordan domain. Its boundary consists of two closed Jordan arcs with common endpoints. The one in $\,\mathfrak X_{\nu+1}\,$  is  denoted by $\,\Gamma_\kappa = \overline{\mathbb X}_\kappa\cap \mathfrak X_{\nu+1} \,$ and the open arc in $\mathbb X\,$ is denoted by $\gamma_\kappa \deff \partial \mathbb X_\kappa \cap \mathbb X\,$. It is at this point that we  take advantage of the condition that the endpoints of $\mathcal C_\kappa$ are simple values of $\, h^{\nu}\,$. This condition implies that the inverse map $\,(h^{\nu} )^{-1} \,: \mathbf I_\kappa \onto \overline{\gamma_\kappa} \,$ is a homeomorphism. On the other hand the preimage  $\,(h^{\nu} )^{-1}(\mathcal C_\kappa)\,\onto  \Gamma_\kappa \subset \mathfrak X_{\nu + 1}\,$   is a closed arc, because of monotonicity of $\,h^{\nu} :\partial\,\mathbb X \onto \partial \,\mathbb Y\,$. Therefore, the open Jordan arc $\,\gamma_\kappa \subset \mathbb X\,$ and  the closed  arc $\, \Gamma_\kappa \subset \mathfrak X_{\nu+1}\,$ form a closed Jordan curve; precisely, the boundary of $\,\mathbb X_\kappa\,$. In summary,
\begin{itemize}
\item $
 \partial\,\mathbb X_\kappa = \gamma_\kappa \cup \Gamma_\kappa\, ,\;\;\;\; h^{\nu} : \mathbb X_\kappa \onto \mathbb S_\kappa\,,
$
 \item
 $ h^{\nu} : \partial\,\mathbb X_\kappa \onto \partial \,\mathbb S_\kappa\,\;\;\textnormal{is continuous and monotone}$
 \item This latter boundary map is injective on the compact subset\\ $\,\mathfrak X_{\nu+1}^\kappa \deff ( \mathfrak X \cap \mathfrak X_{\nu+1} ) \cup \overline{\gamma_\kappa} \subset  \partial\,\mathbb X_\kappa\,$
 \end{itemize}
  Now we appeal to Theorem \ref {RKC} of Rad\'o-Kneser-Choquet  which allows us to replace  $\,h^{\nu} : \mathbb X_\kappa \onto \mathbb S_\kappa\,$ by the harmonic extension of its boundary map $\,h^{\nu} : \partial \,\mathbb X_\kappa \onto \partial \,\mathbb S_\kappa\,$. We need to introduce, for a little while,  more notation.

  \begin{itemize}
   \item $\,h_\kappa^{\nu} : \overline{\mathbb X}_\kappa \onto \overline{\mathbb S}_\kappa\,$ \;\;\;\;\; --- harmonic extension of $\,h^{\nu} : \partial \,\mathbb X_\kappa \onto \partial \,\mathbb S_\kappa\,$
   \item  $\, \textbf{h}^{\nu}  :\mathbb X \onto \mathbb Y \,$ \;\;\;\;\;--- homeomorphism of  class $\mathscr W^{1,2} (\mathbb X ,\mathbb Y)$, defined by
\[
\textbf{h}^{\nu} = 
\begin{cases} 
h^{\nu}_\kappa &  \textrm{on $\, \mathbb X_\kappa\,,\;\;\;\;\kappa = 1, 2, \dots, N $ }\\
h^{\nu} & \textrm{otherwise}
\end{cases}
\]
 \end{itemize}
The continuous extension
$\,\textbf{h}^{\nu} : \overline{\mathbb X} \onto \overline{\mathbb Y} \,$ agrees with $\,h^{\nu} \,$ on $\, \partial\,\mathbb X\,$ and on the compact $\mathbb G\subset \mathbb X\,$ as well.

 Let us estimate the difference $\,h^{\nu} - \textbf{h}^{\nu}\,$ in the norm of Royden algebra. First we find that
$$
\|h^{\nu} - \textbf{h}^{\nu}\|_{\mathscr C^0(\mathbb X)} \leqslant \sup_{1\leqslant \kappa \leqslant N} \|h^{\nu} - \textbf{h}^{\nu}\|_{\mathscr C^0(\mathbb X_\kappa)} \leqslant \sup_{1\leqslant \kappa \leqslant N} \textnormal{diam}\,\mathbb S_\kappa\,  \preccurlyeq \varepsilon
$$
Secondly, in view of  (\ref{Energy}), we see that
$$ \mathscr E_{_\mathbb X} [ h^{\nu} - \textbf{h}^{\nu} ] =  \sum_{\kappa = 1}^{N} \mathscr E_{_{\mathbb X_\kappa} }[ h^{\nu} - h^{\nu}_\kappa ]\,\leqslant   \sum_{\kappa = 1}^{N} \mathscr E_{_{\mathbb X_\kappa} }[ h^{\nu}] \,\preccurlyeq \varepsilon^{\,2}
$$
by  (\ref{ENergy}). Hence
$$
\| h^{\nu} - \textbf{h}^{\nu} \|_{\mathscr R(\mathbb X)} \preccurlyeq \varepsilon
$$
Summarizing, the construction of $\,h^{\nu+1}\,$ in an $\,\varepsilon\,$ proximity to $\,h^\nu\,$ will be done once the similar construction is in hand for  $\,\textbf{h}^{\nu}\,$.
In what follows, instead of using $\,\textbf{h}^{\nu}\,$,  we assume that the original map $\,h^{\nu} \,$ was already harmonic in every $\,\mathbb X_\kappa\,$. This simplifies writing and causes no loss of generality. \\

\textbf{Step III.} (\textit{Reduction of the domain to the unit disk}).

The idea is to construct for each $\,\kappa = 1, 2, \dots , N\,$  a sequence of homeomorphisms $\, h^{\kappa, \,\nu}_j : \overline{\mathbb X}_\kappa \onto \overline{\mathbb S}_\kappa\,, \;j = 1, 2 , \dots \,,\,$ that converge to $\, h^{\nu} : \overline{\mathbb X}_\kappa \onto \overline{\mathbb S}_\kappa\,$ uniformly and in $\,\mathscr W^{1, 2}(\mathbb X_\kappa , \,\mathbb S_\kappa)\,$. In addition to that, we require that each   $\, h^{\kappa, \,\nu}_j\,$ agrees with $\, h^{\nu}\,$ on the compact subset $\,\mathfrak X_{\nu+1}^\kappa \deff ( \mathfrak X \cap \mathfrak X_{\nu+1} ) \cup \overline{\gamma_\kappa} \subset  \partial\,\mathbb X_\kappa\,$. Recall that the boundary map $ h^{\nu} : \partial\,\mathbb X_\kappa \onto \partial \,\mathbb S_\kappa\,$ is injective on $\,\mathfrak X_{\nu+1}^\kappa \,$. Once this is done, the construction of $\, h^{\nu+1}\,$ will be completed in the following way:
for each $\,\kappa\,$ we choose and fix $\, j =  j_\kappa\,$ sufficiently large so that
$$
\| h^{\nu} - h^{\kappa, \,\nu}_{j_\kappa} \|_{\mathscr R(\mathbb X_\kappa)} \le \varepsilon
$$
Then we replace each $\, h^{\nu} : \overline{\mathbb X}_\kappa \onto \overline{\mathbb S}_\kappa\,$ by  $\, h^{\kappa, \,\nu}_{j_\kappa} : \overline{\mathbb X}_\kappa\onto \overline{\mathbb S}_\kappa\,$ to obtain the desired map
\[
h^{\nu+1} \deff
\begin{cases} \, h^{\kappa, \,\nu}_{j_\kappa}  &   \textrm{on $\, \mathbb X_\kappa\,,\;\;\;\;\kappa = 1,2, \dots , n$ } \\
h^{\nu} & \textrm{otherwise}
\end{cases}
\]

Thus we are reduced to finding the sequence $\, h^{\kappa, \,\nu}_j : \overline{\mathbb X}_\kappa\onto \overline{\mathbb S}_\kappa\,, \;j = 1, 2 , \dots \,$. Before proceeding to somewhat involved computation we need to simplify the domain and the target of $\, h^{\nu} : \overline{\mathbb X}_\kappa \onto \overline{\mathbb S}_\kappa\,$. Since the problem is clearly unaffected by a rotation of the target (harmonicity of the map is not compromised), we may confine ourselves to the segment of the form
$$
 \mathbb S_\kappa = \mathbb S \deff \{ \xi\,: |\xi| < 1\,,\;\; \cos \omega <\Re e \,\xi < 1\, \}  ,\;\;\;\textnormal{for some} \;\;0 < \omega < \frac{\pi}{2}
$$
Thus its arc $\,  \mathcal C$ is $\{ \xi = e^{\,i\,\phi} \,: \;\; -\omega \leqslant \phi \leqslant  \omega\,\}\,$, the base $\,\mathbf I $ is $ \{\,\xi =  \cos \omega \,+\,i \tau\, , \;\;\;-\sin \omega \leqslant  \tau \leqslant \sin \omega\, \}\,$ and the corners are $\,\xi^+ = e^{\,i\,\omega}\,,\, \xi^- = e^{\,- i\,\omega} $.
Regarding the domain $\,\mathbb X_\kappa\,$, it is legitimate to conformally transform it onto the unit disk $\mathbb D\,$; any conformal mapping between two Jordan domains induces an isometry of their Royden algebras. Thus we consider a conformal map $\,\chi\,: \mathbb D \, \onto \mathbb X_\kappa\,$ and the pullback  $\,f \deff h^{\nu} \circ \chi : \mathbb D \onto \mathbb S\,$. This is  a homeomorphism up to the boundaries.  One extra condition turns out to be useful; namely, we may choose $\,\chi\,$ to be normalized  at three boundary points so that the map $\,f \deff h^{\nu} \circ \chi\,$ satisfies,
\begin{equation}\label{Bconditions}
f(e^{\,i\,\omega}) = e^{\,i\,\omega}\,,\;\;\; f(e^{\,-i\,\omega}) = e^{\,-i\,\omega}\,,\;\;\;\; f(-1) = \cos \omega\,\;
\end{equation}
The first two values of $\,f\,$ are the endpoints of the base $\,\mathbf I \subset \partial\,\mathbb S\,$ and last one is its midpoint. Let us state clearly what we aim to show.

\begin{proposition}\label{propo} Let $\,f : \mathbb D \onto \mathbb S\,$ be an orientation preserving  harmonic homeomorphism of finite energy and let its continuous extension $\,f : \overline{\mathbb D} \onto \overline{\mathbb S}\,$ satisfy~\eqref{Bconditions}. Thus $f$ maps the closed  arc $\mathcal C\,$ monotonically onto itself and  $\,\mathbf T \,$ homeomorphically onto the base $\,\mathbf I \subset \partial\,\mathbb S\,$. Suppose, in addition,  that  $\,f : \mathcal C \onto \mathcal C\,$ is  injective on a compact subset $\,\mathbf K \subset \mathcal C\,$. Then there exist homeomorphisms $\,f_m \,: \overline{\mathbb D} \onto \overline{\mathbb D}\,$ converging to $\,f : \overline{\mathbb D} \onto \overline{\mathbb S}\,$ uniformly and in the Sobolev space $\,\mathscr W^{1, 2}(\mathbb D , \mathbb D)\,$. Moreover $\,f_m \equiv f\,$ on $\,\mathbf T \,$ and $\,\mathbf K\,$, for $\, m= 4, 5, \dots$.
\end{proposition}

Once this proposition is established, the proof of Theorem \ref{Main} is complete. Thus, in steps IV and V, we shall concern ourselves only with the proof of this proposition.

\begin{center}\begin{figure}[h]\label{figure2}
\includegraphics[width=0.99\textwidth]{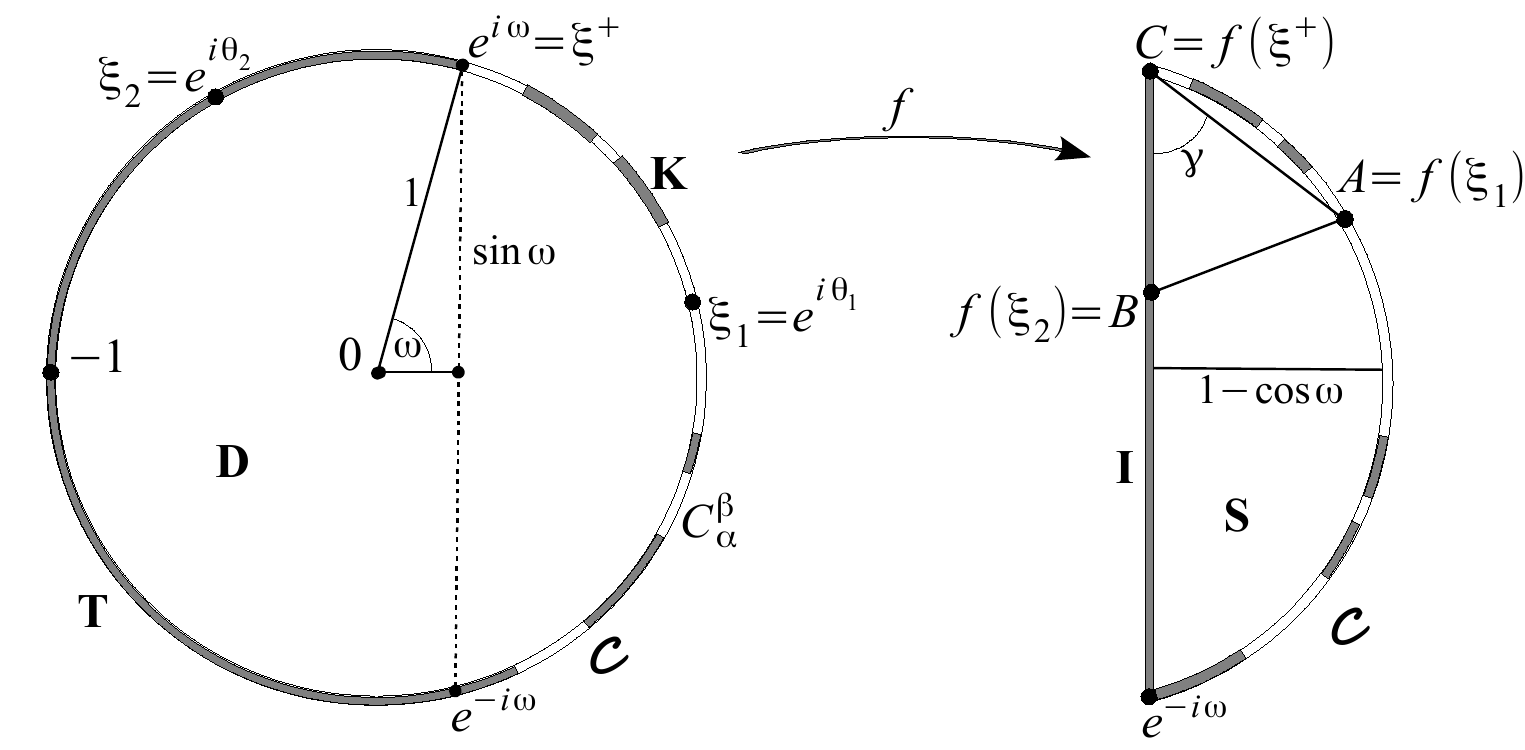} \caption{Harmonic map $f \colon \mathbb D \onto \mathbb S$.}\label{seg}
\end{figure}\end{center}

\textbf{Step IV.} (\textit{Good approximation of the boundary map $\,f : \partial\,\mathbb D \onto \partial\,\mathbb S\,$}).
The aim is to properly approximate the continuous monotone boundary map $\,f : \partial\,\mathbb D \onto \partial\,\mathbb S\,$  by homeomorphisms $\,f_m : \partial\,\mathbb D \onto \partial\,\mathbb S\,$  which agree with $\, f\,$ on both sets $\,\mathbf T = \overline{\mathbb T \setminus \mathcal C}\,$ and $\,\mathbf K\subset \mathcal C\,$. Once such an approximation is in hand, we shall extend each $\,f_m\,$ harmonically inside the disk. Certainly, uniform convergence $\,f_m \rightrightarrows f\,$ on $\,\mathbb T\,$ would suffice to deduce uniform convergence in the entire disk, by the maximum principle. However, it is not obvious at all how to make the approximation of the boundary map in order to control the energy of the extended mappings. The Douglas criterion will come into play.

We need only work to construct  homeomorphisms $\,f_m  : \mathcal C \onto \mathcal C\,$, as their values on $\, \mathbf T = \overline{\mathbb T\setminus \mathcal C}\,$ are already known; $\, f_m \,\equiv\, f  : \mathbf T \onto \,\mathbf T\,$.
 Let us write $\,f : \mathcal C \onto \mathcal C\,$ as $\,f(e^{\,i\,\theta}) =  e^{\,i\,\phi(\theta)}\,$, where $\,\phi : [-\omega\,, \omega ] \onto [-\omega\,, \omega ]\,$ is a nondecreasing continuous function such that $\,\phi(-\omega) = -\omega\,$ and $\,\phi(\omega) = \omega\,$. We also require that $\, f_m \,\equiv\, f \,$ on $\,\mathbf K \subset \mathcal C\,$.  The complement $\,\mathcal C \setminus \mathbf K\,$ consists of a countable number of disjoint open circular arcs (components) whose endpoints belong to $\,\mathbf K\,$. Let $\, \mathcal C^\beta_\alpha \, \deff \,\{e^{i\,\theta}\,;\; \alpha \leqslant \theta \leqslant \beta  \} \,$ be one of such arcs together with its endpoints, where $\, -\omega \leqslant \alpha < \beta \leqslant \omega\,$, and similarly for the image $\, -\omega \leqslant \phi(\alpha) < \phi(\beta) \leqslant \omega\,$. The latter strict inequality is justified by the fact that $\, e^{\,i\,\alpha}\,$ and  $\, e^{\,i\,\beta}\,$ are points in $\mathbf K\,$ and  the map $\,f\,$, being injective on $\,\mathbf K\,$,  assumes  distinct values $\, e^{\,i\,\phi(\alpha)}\,$ and  $\, e^{\,i\,\phi(\beta)}\,$ at these points.

 We now define homeomorphisms between closed arcs $\,f_m  : \mathcal C^{\,\beta}_\alpha  \onto \mathcal C^{\,\phi(\beta)}_{\phi(\alpha)} \,$ ($\,m \geqslant 4\,$)  by the rule; $\,f_m(e^{\,i\,\theta}) = e ^{\,i\, \phi_m(\theta)}\,$, where
 \begin{equation}\label{Deff/of/fm}
 \;\;\phi_m(\theta) = \Big[1 - \frac{\beta - \alpha}{m} \Big]\cdot \big[ \phi(\theta) - \phi(\alpha) \big]\;+\; \frac{\phi(\beta) - \phi(\alpha)}{m}\,\cdot \big(\theta - \alpha\big) \;+\; \phi(\alpha)
 \end{equation}
 for $\,\alpha \leqslant \theta \leqslant \beta\,$.
 We have $\, \phi_m(\alpha) = \phi(\alpha)\,$ and  $\, \phi_m(\beta) = \phi(\beta)\,$. The first term defining $\,\phi_m\,$ is nondecreasing in $\,\theta\,$ while the second term is strictly increasing. Therefore, $\,f_m  : \mathcal C^{\,\beta}_\alpha  \onto \mathcal C^{\,\phi(\beta)}_{\phi(\alpha)} \,$ is a homeomorphism. Formula (\ref{Deff/of/fm}), applied to every arc component of $\,\mathcal C \setminus \mathbf K\,$,  gives homeomorphisms which agree with $\,f\,$ at the endpoints of the arcs. We glue them together with $\,f\,$ at the endpoints to obtain a homeomorphism $\, f_m  : \mathbb T \onto \mathbb T\,$ which coincides with $\, f  : \mathbb T \onto \mathbb T\,$ on $\,\mathbf T \cup \mathbf K\,$. Further analysis of  $\,f_m\,$ is necessary to deduce proper convergence  as $\, m \rightarrow \infty\,$.
First note that on each arc component of $\,\mathcal C \setminus \mathbf K\,$ we have
$$\,|\,f_m(e^{\,i\,\theta}) - f(e^{\,i\,\theta})\,| \leqslant |\,\phi_m(\theta) -\phi(\theta)\,|\, \leqslant \frac{8\,\omega ^2}{m}  < \frac{21}{m}
$$
Hence
\begin{equation}\label{uniform}
|\,f_m(\xi) - f(\xi)\,|\leqslant  \frac{21}{m}\,,\;\;\;\textnormal{for every}\;\;\xi \in \mathbb T \, \;\;\textnormal{and}\;\; m = 4,5, \dots
\end{equation}
In particular, $\,f_m \rightrightarrows f\,$ uniformly on $\,\mathbb T\,$.

\begin{lemma}\label{mess} For all $\,\xi_1 ,\xi_2 \in \mathbb T\,$ and $\, m = 4,5, \dots$, we have
\begin{equation}\label{Estimate for Douglas}
|f_m(\xi_1) - f_m(\xi_2)| \;\leqslant \frac{5}{\sin (\omega/4)}  \, |f(\xi_1) - f(\xi_2)|\;+\;  4\,|\xi_1 - \xi_2 |
\end{equation}
\end{lemma}

\begin{proof} There are three cases to consider:

 \textit{\textbf{Case 1.}} We first do the case when  both $\,\xi_1\,$ and $\,\xi_2\,$ belong to the closure of the same component  of $\,\mathcal C\setminus \mathbf K\,$; say, $\,\xi_1 = e^{\,i\,\theta_1} \in \mathcal C^{\,\beta}_\alpha\,$  and $\,\xi_2 = e^{\,i\,\theta_2} \in \mathcal C^{\,\beta}_\alpha\,$,  where $\,\alpha \leqslant \theta_1 \,,\, \theta _2  \leqslant \beta\,$.  It follows from formula (\ref{Deff/of/fm})   that
\[
\begin{split}
\,| \,f_m(e^{\,i\,\theta_1}) - f_m(e^{\,i\,\theta_2})\,| &\leqslant   | \,\phi_m(\theta_1) - \phi_m(\theta_2)\,|  \leqslant  | \,\phi(\theta_1) - \phi(\theta_2)\,| \,+ |\theta_1 - \theta_2|
\\
&\leqslant 2\,| \,f(e^{\,i\,\theta_1}) - f(e^{\,i\,\theta_2})\,| \;+ 2\,| \,e^{\,i\,\theta_1} - e^{\,i\,\theta_2}\,|
\end{split}
\]
\vskip0.3cm
 \textit{\textbf{Case 2.}} Both $\,\xi_1 = e^{\,i\,\theta_1}\,$ and $\,\xi_2 = e^{\,i\,\theta_2}\,$ belong to  $\,\mathcal C\,$. We assume that the closed set  $\,\{\, \tau\;:\; \theta_1 \leqslant \tau \leqslant \theta_2\,,\;\;\;e^{\,i\,\tau\,} \in \mathbf K\,\}\,$ is not empty. Otherwise, $\,\xi_1\,$ and $\,\xi_2\,$ would belong to the same arc component of  $\,\mathcal C\setminus \mathbf K\,$. Let us set the notation,
 $$
 \tau_1 \deff \min \,\{\, \tau\;:\; \theta_1 \leqslant \tau \leqslant \theta_2\,,\;\;\;e^{\,i\,\tau\,} \in \mathbf K\,\} \,,\;\;\;\; \widehat{\xi}_1 \deff e^{\,i\,\tau_1} \in \mathbf K\;$$
$$ \tau_2 \deff \max\,\{\, \tau\;:\; \theta_1 \leqslant \tau \leqslant \theta_2\,,\;\;\;e^{\,i\,\tau\,} \in \mathbf K\,\}  \,,\;\;\;\; \widehat{\xi}_2\deff e^{\,i\,\tau_2} \in \mathbf K\;
 $$
 and note that the points $\,\xi_1\,,\widehat{\xi}_1\,$ belong to the closure of one arc component in $\,\mathcal C\setminus \mathbf K\,$. The same applies to the pair $\,\xi_2\,,\widehat{\xi}_2\,$. Therefore, using \textit{Case 1}, one can write the following chain of inequalities

 \begin{equation}\label{equicont}
\begin{split}
|\,f_m(\xi_1) &- f_m(\xi_2)\,| 
\\ & \leqslant  |\,f_m(\xi_1) - f_m(\widehat{\xi}_1)\,|\,\, + \,\,|\,f_m(\widehat{\xi}_1) - f_m(\widehat{\xi}_2)\,|\,\, + \,\, |\,f_m(\widehat{\xi}_2) - f_m(\xi_2)\,| \nonumber \;\\& \leqslant  2\,|\,f(\xi_1) - f(\widehat{\xi}_1)\,|\,\, + \,\,2\, |\, \xi_1 - \widehat{\xi}_1\,|\,\, +\,\, |\,f(\widehat{\xi}_1) - f(\widehat{\xi}_2)\,|\;  \\& + 2\,|\,f(\widehat{\xi}_2) - f(\xi_2)\,|\,\, +\,\, 2\,|\, \widehat{\xi}_2 -  \xi_2\,|
 \end{split}
 \end{equation}
 Since $\,\widehat{\xi}_1\,$ and $\,\widehat{\xi}_2\,$ lie in the shorter circular arc that connects $\,\xi_1\,$ to $\,\xi_2\,$, it follows that $$\, |\,\xi_1 - \widehat{\xi}_1 \,| \leqslant |\,\xi_1 - \xi_2 \,|\,\;\;\;\textnormal{and}\;\;\;\, |\,\widehat{\xi}_2 - \xi_2\, | \leqslant |\,\xi_1 - \xi_2 \,|\,
 $$
 The same argument applies to the images of these point under the monotone map $\, f  : \mathcal C \onto \mathcal C\,$. Thus,
  $$\, |\,f(\xi_1) - f(\widehat{\xi}_1) \,| \leqslant |\,f(\xi_1) - f(\xi_2)\, |\,\;\;\;\textnormal{and}\;\;\;\, |\,f(\widehat{\xi}_2) - f(\xi_2) \,| \leqslant |\,f(\xi_1) - f(\xi_2) \,|\,
  $$
  and also  $\,|\,f(\widehat{\xi}_1) - f(\widehat{\xi}_2) \,| \leqslant |\,f(\xi_1) - f(\xi_2)\, |\,$.
  Substitute these inequalities into the chain above to conclude with the desired inequality
  $$
  |\,f_m(\xi_1) - f_m(\xi_2)\,| \leqslant \;5\,|\,f(\xi_1) - f(\xi_2)\,| \;+\;4\,|\,\xi_1 -  \xi_2\,|
  $$

The case $\,\xi_1 \,,\, \xi_2 \in \mathbf T \,$ is trivial, because $\,f_m(\xi_1) - f_m(\xi_2) = f(\xi_1) - f(\xi_2)\,$.  Thus, all that remains is to consider

\textit{\textbf{Case 3.}} Let $\,\xi_1 \in \mathcal C\,$ and $\,\xi_2\in \mathbf T \,$.  By symmetry we may take $\, \xi_1 = e^{\,i\,\theta_1}\,$,  where $\, 0\leqslant \theta_1 < \omega\,$.  We may also assume that $|\,\xi_1 - \xi_2\,|  \leqslant  \sin {\omega}$. Otherwise, inequality  (\ref{Estimate for Douglas}) holds; namely, $\,|\,f_m(\xi_1) - f_m(\xi_2)\,| \leqslant \textnormal{diam} \,\mathbb S = 2\,\sin \omega  < 2 |\,\xi_1 - \xi_2\,|\,$. Geometrically, the assumption $\,|\,\xi_1 - \xi_2\,|  \leqslant  \sin {\omega}$ tells us that $\,\xi_2\,$ cannot lay in the lower half of the arc $\, \mathbf T\,$. Thus $\,\xi_2 = e^{\,i\,\theta_2}\,$,  where $\, \omega < \theta_2 \leqslant \pi\,$. Let the upper corner of the segment $\,\mathbb S\,$ be denoted by  $\, \xi^+ = e^{\,i\,\omega}\,$.  The location of $\,f(\xi_2)\,$ is restricted to the upper half of the base of the segment $\,\mathbb S\,$, because $\,f : \mathbb T \onto \mathbb T \,$ is monotone and  $\,f( e^{\,i\,\pi})  = \cos \omega\,$  (the midpoint  of the base)  due to normalization at (\ref{Bconditions}). Regarding the position of $\,f(\xi_1) \in \mathcal C\,$, we may assume that this value  also lies in the upper half of the arc $\,\mathcal C\,$. Otherwise, we would have  $\,|\,f(\xi_1) - f(\xi_2)\,|  \geqslant  1 - \cos \omega  = 2 \sin^2 \frac{\omega}{2}\,$ while, on the other hand, 
\[
\begin{split}
\,|\,f_m(\xi_1) - f_m(\xi_2)\,| &\leqslant \textnormal{diam} \,\mathbb S = 2\,\sin \omega \, <  \frac{5}{\sin (\omega/4)}\cdot  2 \sin^2 \frac{\omega}{2} \\ &\leqslant \frac{5}{\sin (\omega/4)} \; |\, f(\xi_1)\, -\,f(\xi_2) \,|\,  
\end{split}
\]
which implies~\eqref{Estimate for Douglas}.

We are ready to complete Case~3. First we use Case~2 and the triangle inequality,  
\begin{equation}\label{equicont222}
\begin{split}
|\,f_m(\xi_1) - f_m(\xi_2)\,| \;& \leqslant  |\,f_m(\xi_1) - f_m(\xi^+)\,|\,\, + \,\,|\,f_m(\xi^+) - f_m(\xi_2)\,| \nonumber \;\\& \leqslant  5\,|\,f(\xi_1) - f(\xi^+)\,|\,\, + \,\,4\, |\, \xi_1 - \xi^+\,|\,\, +\,\, |\,f(\xi^+) - f(\xi_2)\,|\;  \\&  \leqslant  5\,\big \{\,|\,f(\xi_1) - f(\xi^+)\,|\,+\,  |\,f(\xi^+) - f(\xi_2)\,|\,\big\}+ \,4\,|\, \xi_1-  \xi_2\,|
 \end{split}
 \end{equation}
Then comes a geometric fact about the term within the curled braces. Certainly, we have $\,f(\xi^+) \neq f(\xi_2)\,$, because $\,f\,$ is injective on $\,\mathbf T\,$. If, incidentally,  $\,f(\xi^+) =  f(\xi_1)\,$  then the latter estimate yields~\eqref{Estimate for Douglas}. Thus we may assume that three points $\,A \deff f(\xi_1)\,, B\deff f(\xi_2)\,$ and $\, C \deff f(\xi^+) \,$ are vertices of a triangle. Let $ a = |\, B - C\,| , \, b = |\, A - C\,| \,$ and $\, c = |\,A - B\,|\,$. Since $\,A\,$ lies in the arc of $\,\mathbb S\,$, $\, B \,$ lies in the base of $\,\mathbb S\,$ and $\,C \,$ is the corner of $\mathbb S\,$, all of them in the upper half of $\,\mathbb S\,$, it follows (from geometry of the segment $\,\mathbb S\,$) that the angle opposite to the side $\,\overline{A\, B}\,$, denoted by $\,\gamma\,$, satisfies: $\, \frac{\omega}{2}\leqslant  \gamma < \omega\,$. The law of cosines tells us that 
 \[
 c^2 = a^2 + b^2 - 2 ab \,\cos \,\gamma\,\geqslant a^2 + b^2 - 2 ab \,\cos \,(\omega/2) \geqslant  \,(a\,+\,b\,)^2\,\sin^2\frac{\omega}{4}.
\]
  Hence
\[
|\,f(\xi_1) - f(\xi^+)\,|\,+\,  |\,f(\xi^+) - f(\xi_2)\,|\, \leqslant \frac{1}{\sin\frac{\omega}{4}} \,|\,f(\xi_1) - f(\xi_2)\,| \,,
\] 
completing the proof of Lemma~\ref{mess}.
\end{proof}

\textbf{Step V.} (\textit{Harmonic extension and strong convergence in}  $\,\mathscr W^{1,2}(\mathbb D)\; $).

 The boundary homeomorphisms  $\,\,f_m  \,: \,\mathbb T  \onto \partial\,\mathbb S\,$ will now be extended harmonically inside the unit disk. We use the same label  for the extensions, $\,\,f_m  \,: \, \,\overline{\mathbb D} \onto \,\overline{\mathbb S }\,$. These mappings are homeomorphisms, due to Theorem~\ref{RKC}.   Since both $\,f\,$ and $\,f_m\,$ are harmonic in $\,\mathbb D\,$, the sequence $\,f_m\,$ converges to $\,f\,$ uniformly in $\,\overline{\mathbb D}\,$, by the maximum principle.  The key point here is that they also belong to the Sobolev space $\,\,\mathscr W^{1,2}(\mathbb D) \,$, and converge in the Sobolev norm as well. To see this we recall the Douglas  criterion~\cite{Do} which asserts that any function  $\,g\,$ that is continuous on $\,\overline{\mathbb D}\,$ and harmonic in $\,\mathbb D\,$ satisfies
 \begin{equation}\label{Douglas}
\mathscr E_{_\mathbb D} [g] \,\deff \, \iint_\mathbb D |D g|^2 \;=\, \frac{1}{ 2\,\pi} \underset {\mathbb T \,\times \,\mathbb T }{\;\,\;\int \,\int}\;  \Big{|}\frac{g(\xi) - g(\zeta)}{\xi - \zeta }\Big{|}^2 \; |\textnormal d\xi| \!\cdot\!|\textnormal{d}\zeta | \;.
 \end{equation}
 %We know that for $f\,: \,\mathbb T  \onto \partial\,\mathbb S\,$ the integrals in~\eqref{Douglas} are finite. 
 Recall that by~\eqref{Estimate for Douglas} the mappings  $\,\,f_m  \,: \,\mathbb T  \onto \partial\,\mathbb S\,$ satisfy
\[
  \Big{|}\frac{f_m(\xi) - f_m(\zeta)}{\xi - \zeta }\Big{|}^2 \;\preccurlyeq\;  \Big{|}\frac{f(\xi) - f(\zeta)}{\xi - \zeta }\Big{|}^2 \; +\; 1
 \]
 where the implied constant does not depend on $\,m\,$.
 By virtue of~\eqref{Douglas} this implies
 \[
 \mathscr E_{_\mathbb D} [f_m] \;\preccurlyeq\;  \mathscr E_{_\mathbb D} [f]+ 1 <\infty\;.
 \]
 Therefore $\, f_m\,$ have uniformly bounded energy. It follows that $\,f_m\,$ converge to $\,f\,$ not only uniformly but also weakly in $\,\mathscr W^{1,2}(\mathbb D)\,$. In particular, $\,\mathscr E[f] \leqslant \liminf \mathscr E[f_m]\,$. It is crucial to notice, using Dominated Convergence Theorem, that in fact we have equality
   \begin{equation}
   \begin{split}
   \mathscr E_{_\mathbb D}[f] \,&= \frac{1}{ 2\,\pi} \underset {\mathbb T \,\times \,\mathbb T }{\;\,\;\int \,\int}\;  \Big{|}\frac{f(\xi) - f(\zeta)}{\xi - \zeta }\Big{|}^2 \; |\textnormal d\xi| \!\cdot\!|\textnormal{d}\zeta | \\&=  \lim \frac{1}{ 2\,\pi} \underset {\mathbb T \,\times \,\mathbb T }{\;\,\;\int \,\int}\;  \Big{|}\frac{f_m(\xi) - f_m(\zeta)}{\xi - \zeta }\Big{|}^2 \; |\textnormal d\xi| \!\cdot\!|\textnormal{d}\zeta | = \lim \mathscr E_{_\mathbb D}[f_m].
   \end{split}
   \end{equation}

This shows that $\,f_m\,$ converge to $\,f\,$ strongly in $\,\mathscr W^{1,2}(\mathbb D)\,$, completing the proof of 
Proposition~\ref{propo} and thus of Theorem \ref{Main}\,.
\qed

 \bibliographystyle{amsplain}

 \end{document}